\documentclass[12pt]{article}
\usepackage{mathrsfs}

\usepackage{t1enc}
\usepackage[latin1]{inputenc}
\usepackage[english]{babel}
\usepackage{amsmath,amsthm}
\usepackage{amsfonts}
\usepackage{latexsym}
\usepackage{graphicx}
\usepackage[natural]{xcolor}
\theoremstyle{plain}
\newtheorem{theorem}{Theorem}

\newtheorem{lemma}{Lemma}[section]

\newtheorem{claim}{Claim}
\newtheorem{case}{Case}
\newtheorem{subcase}{Subcase}[case]
\newtheorem{fact}{Fact}

 \renewenvironment{proof}{
 \noindent{\bf Proof.}\rm} {\mbox{}\hfill\rule{0.5em}{0.809em}\par}

\usepackage{caption,subcaption}

\begin{document}
\date{}
\title{Gallai-Ramsey numbers for monochromatic $K_4^{+}$ or $K_{3}$
}
\author{\small Xueli Su, Yan Liu\\
\small School of Mathematical Sciences, South China Normal University,\\
 \small Guangzhou, 510631, P.R. China  \thanks{This work is supported
by the Scientific Research Fund of the Science and Technology Program of Guangzhou, China (authorized in 2019) and by the Qinghai Province Natural Science Foundation£¨No.2020-ZJ-924).
 Correspondence should be addressed to Yan
Liu(e-mail:liuyan@scnu.edu.cn)}}

\maketitle

\setcounter{theorem}{0}

\begin{abstract}
A Gallai $k$-coloring is a $k$-edge coloring of a complete graph in which
there are no rainbow triangles. For two given graphs $H, G$
and two positive integers $k,s$ with that $s\leq k$, the
$k$-colored Gallai-Ramsey number $gr_{k}(K_{3}: s\cdot H,~
(k-s)\cdot G)$ is the minimum integer $n$ such that every Gallai
$k$-colored $K_{n}$ contains a monochromatic copy of $H$ colored
by one of the first $s$ colors or a monochromatic copy of $G$
colored by one of the remaining $k-s$ colors. In this paper, we
determine the value of Gallai-Ramsey number in the case that
$H=K_{4}^{+}$ and $G=K_{3}$. Thus the Gallai-Ramsey number
$gr_{k}(K_{3}: K_{4}^{+})$ is obtained.

\end{abstract}

\vspace{4mm}
\section{Introduction}

All graphs considered in this paper are finite, simple and undirected. For a graph $G$, we use $|G|$ to denote the number of vertices of $G$, say the \emph{order} of $G$.
The complete graph of order $n$ is denoted by $K_{n}$.
For a subset $S\subseteq V(G)$, let $G[S]$ be the subgraph of $G$ induced by $S$.
For two disjoint subsets $A$ and $B$ of $V(G)$, $E_{G}(A, B)=\{ab\in E(G) ~|~ a\in A, b\in B\}$. Let $G_1=(V_1,E_1)$ and $G_2=(V_2,E_2)$ be two graphs. The \emph{union} of $G_1$ and $G_2$, denoted by $G_{1}+G_{2}$, is the graph with the vertex set $V_1\bigcup V_2$ and the edge set $E_1\bigcup E_2$. The \emph{join} of $G_1$ and $G_2$, denoted by $G_1\vee G_2$, is the graph obtained from $G_1+G_2$ by adding all edges joining each vertex of $G_1$ and each vertex of $G_2$.
For any positive integer $k$, we write $[k]$ for the set $\{1, 2, \cdots, k\}$.
An edge coloring of a graph is called \emph{monochromatic} if all edges are colored by the same color. An edge-colored graph is called \emph{rainbow} if no two edges are colored by the same color.

Given graphs $H_{1}$ and $H_{2}$, the classical Ramsey number $R(H_{1}, H_{2})$ is the smallest integer $n$ such that for any 2-edge coloring of $K_{n}$ with red and blue, there exists a red copy of $H_{1}$ or a blue copy of $H_{2}$.
A \emph{sharpness example} of the Ramsey number $R(H_{1}, H_{2})$, denoted by $C_{(H_{1}, H_{2})}$, is a 2-edge colored $K_{R(H_{1}, H_{2})-1}$ with red and blue such that there are neither red copies of $H_{1}$ nor blue copies of $H_{2}$. Given graphs $H_{1}, H_{2}, \cdots, H_{k}$, the multicolor Ramsey number $R(H_{1}, H_{2}, \cdots, H_{k})$ is the smallest positive integer $n$ such that for every $k$-edge colored $K_{n}$ with the color set $[k]$, there exists some $i\in [k]$ such that $K_{n}$ contains a monochromatic copy of $H_{i}$ colored by $i$.
The multicolor Ramsey number is an obvious generalization of the classical Ramsey number.
When $H=H_{1}=\cdots=H_{k}$, we simply denote $R(H_{1}, \cdots, H_{k})$ by $R_{k}(H)$. The problem about computing Ramsey numbers is notoriously difficult.
For more information on classical Ramsey number, we refer the readers to ~\cite{PRCRT, RRA, RBJR}.
In this paper, we study Ramsey number in Gallai-coloring. A \emph{Gallai-coloring} is an edge coloring of a complete graph with no rainbow triangle. Gallai-coloring naturally arises in several areas including: information theory \cite{JG}; the study of partially ordered sets, as in Gallai's original paper \cite{Gallai} (his result was restated in \cite{GyarfasSimonyi} in the terminology of graphs); and the study of perfect graphs \cite{KJLA}. More information on this topic can be found in \cite{JAJT, SCKR}.
A Gallai $k$-coloring is a Gallai-coloring that uses at most $k$ colors.
Given a positive integer $k$ and graphs $H_{1}, H_{2}, \cdots, H_{k}$, the \emph{Gallai-Ramsey number} $gr_{k}(K_{3}: H_{1}, H_{2}, \cdots, H_{k})$ is the smallest integer $n$ such that every Gallai $k$-colored $K_{n}$ contains a monochromatic copy of $H_{i}$ in color $i$ for some $i\in [k]$.
Clearly, $gr_{k}(K_{3}: H_{1}, H_{2}, \cdots, H_{k})\leq R(H_{1}, H_{2}, \cdots, H_{k})$ for any $k$ and $gr_{2}(K_{3}: H_{1}, H_{2})=R(H_{1}, H_{2})$.
When $H=H_{1}=\cdots=H_{k}$, we simply denote $gr_{k}(K_{3}: H_{1}, H_{2}, \cdots, H_{k})$ by $gr_{k}(K_{3}: H)$.
When $H=H_{1}=\cdots=H_{s} (0\leq s\leq k)$ and $G=H_{s+1}=\cdots=H_{k}$, we use the following shorthand notation $$gr_{k}(K_{3}:s\cdot H, (k-s)\cdot G)=gr_{k}(K_{3}:\underbrace{H, \cdots, H}_{s ~\text{times}}, \underbrace{G, \cdots, G}_{(k-s) ~\text{times}}).$$

The authors in \cite{HColton} and \cite{WMG} determined the Gallai-Ramsey number $gr_{k}(K_{3}:s\cdot K_{4}, (k-s)\cdot K_{3})$ and $gr_{k}(K_{3}:s\cdot K_{3}, (k-s)\cdot C_{4})$, respectively. In this paper, we investigate the Gallai-Ramsey number $gr_{k}(K_{3}:s\cdot K_{4}^{+}, (k-s)\cdot K_{3}),$ where $K_{4}^{+}=K_{1}\vee(K_{3}+K_{1})$.
We will prove the following result in Section 2.

\begin{theorem}\label{Thm:K4+K3}
Let $k$ be a positive integer and $s$ an integer such that $0\leq s\leq k$. Then
$$
gr_k(K_{3}:s\cdot K_{4}^{+}, (k-s)\cdot K_{3})=\begin{cases}
17^{\frac{s}{2}}\cdot 5^{\frac{k-s}{2}}+1, &\text{if $s$ and $(k-s)$ are both even,}\\
2\cdot17^{\frac{s}{2}}\cdot 5^{\frac{k-s-1}{2}}+1, &\text{if $s$ is even and $(k-s)$ is odd,}\\
4\cdot 17^{\frac{k-1}{2}}+1, &\text{if $s=k$ and $k$ is odd,}\\
8\cdot17^{\frac{s-1}{2}}\cdot 5^{\frac{k-s-1}{2}}+1, &\text{if $s$ and $(k-s)$ are both odd,}\\
16\cdot17^{\frac{s-1}{2}}\cdot 5^{\frac{k-s-2}{2}}+1, &\text{if $s<k$ and $s$ is odd and $(k-s)$ is even.}
\end{cases}
$$
\end{theorem}

When $s=k$ and $s=0$, we can get the following Theorem~\ref{Thm:K4+} and Theorem~\ref{Thm:K3} from Theorem~\ref{Thm:K4+K3}, respectively. So Theorem~\ref{Thm:K4+} and Theorem~\ref{Thm:K3} can be seen as two corollaries obtained from Theorem~\ref{Thm:K4+K3}.
\begin{theorem}\label{Thm:K4+}
For a positive integer $k$,
$$
gr_k(K_{3}:K_{4}^{+})=\begin{cases}
17^{\frac{k}{2}}+1, &\text{if $k$ is even,}\\
4\cdot 17^{\frac{k-1}{2}}+1, &\text{if $k$ is odd.}
\end{cases}
$$
\end{theorem}

\begin{theorem}{\upshape \cite{Fre, AGAS}}\label{Thm:K3}
For a positive integer $k$,
$$
gr_k(K_{3}:K_{3})=\begin{cases}
5^{\frac{k}{2}}+1, &\text{if $k$ is even,}\\
2\cdot 5^{\frac{k-1}{2}}+1, &\text{if $k$ is odd.}
\end{cases}
$$
\end{theorem}
To prove Theorem~\ref{Thm:K4+K3}, the following theorem is useful.

\begin{theorem}{\upshape \cite{Gallai, GyarfasSimonyi, CameronEdmonds}} {\upshape (Gallai-partition)}\label{Thm:G-Part}
For any Gallai-coloring of a complete graph $G$, there exists a partition of $V(G)$ into at least two parts such that there are at most two colors on the edges between the parts and only one color on the edges between each pair of parts. The partition is called a \emph{Gallai-partition}.
\end{theorem}

Given a Gallai-partition $(V_{1}, V_{2}, \cdots, V_{t})$ of a Gallai-colored complete graph $G$, let $H_{i}=G[V_{i}]$, $h_{i}\in V_{i}$ for each $i\in [t]$ and $R=G\left[\{h_{1}, h_{2}, \cdots, h_{t}\}\right]$.
Then $R$ is said to be the \emph{reduced graph} of $G$ corresponding to the given Gallai-partition. By Theorem~\ref{Thm:G-Part}, all edges in the reduced graph $R$ are colored at most two colors.

\section{Proof of Theorem~\ref{Thm:K4+K3}}

First, recall some known classical Ramsey numbers of $K_{4}^{+}$ and $K_{3}$.

\begin{lemma}{\upshape \cite{Clm, HHIM, SPR}}\label{Lem:k4+K3}
\begin{align*}
R(K_{3}, K_{3})=6,~~~  R(K_{4}^{+}, K_{3})=R(K_{3}, K_{4}^{+})=9,~~~ R(K_{4}^{+}, K_{4}^{+})=18.
\end{align*}
\end{lemma}

For the sake of notation, now we define functions $f_{i}$ for $i\in [5]$ as follows.
\begin{align*}
f_{1}(k, s)&=17^{\frac{s}{2}}\cdot 5^{\frac{k-s}{2}},\\
f_{2}(k, s)&=2\cdot17^{\frac{s}{2}}\cdot 5^{\frac{k-s-1}{2}},\\
f_{3}(k, s)&=4\cdot 17^{\frac{k-1}{2}},\\
f_{4}(k, s)&=8\cdot17^{\frac{s-1}{2}}\cdot 5^{\frac{k-s-1}{2}},\\
f_{5}(k, s)&=16\cdot17^{\frac{s-1}{2}}\cdot 5^{\frac{k-s-2}{2}}.
\end{align*}
Let $$
f(k, s)=\begin{cases}
f_{1}(k, s), &\text{if $s$ and $(k-s)$ are both even,}\\
f_{2}(k, s), &\text{if $s$ is even and $(k-s)$ is odd,}\\
f_{3}(k, s), &\text{if $s=k$ and $k$ is odd,}\\
f_{4}(k, s), &\text{if $s$ and $(k-s)$ are both odd,}\\
f_{5}(k, s), &\text{if $s<k$, $s$ is odd and $(k-s)$ is even.}
\end{cases}
$$
It is easy to check that

\begin{eqnarray*}
2f(k-1, s)
&=&\begin{cases}
2f_{2}(k-1, s), &\text{if $s$ and $(k-s)$ are both even,}\\
2f_{1}(k-1, s), &\text{if $s$ is even and $(k-s)$ is odd,}\\
2f_{3}(k-1, s), &\text{if $s=k-1$ and $k$ is even,}    \ \ \ \ \ \ \ \ \ \ \ \ \ \ \ \ \ \ \ \ \ \ \ \ \ \ \ \ \ \ \ \ \ \ \ \ \ \ \ \ (1)\\
2f_{5}(k-1, s), &\text{if $s<k-1$, $s$ and $(k-s)$ are both odd,}\\
2f_{4}(k-1, s), &\text{if $s$ is odd and $(k-s)$ is even}\\
\end{cases}\\
&\leq &f(k, s),
\end{eqnarray*}

\begin{eqnarray*}
5f(k-2, s)
&=&\begin{cases}
5f_{1}(k-2, s), &\text{if $s$ and $(k-s)$ are both even,}\\
5f_{2}(k-2, s), &\text{if $s$ is even and $(k-s)$ is odd,}\\
5f_{3}(k-2, s), &\text{if $s=k-2$ and $k$ is odd,}   \ \ \ \ \ \ \ \ \ \ \ \ \ \ \ \ \ \ \ \ \ \ \ \ \ \ \ \ \ \ \ \ \ \ \ \ \ \ \ \ \ \ \ \ \ \ \ \ \ \ (2)\\
5f_{4}(k-2, s), &\text{if $s$ and $(k-s)$ are both odd,}\\
5f_{5}(k-2, s), &\text{if $s<k-2$ and $s$ is odd and $(k-s)$ is even}\\
\end{cases}\\
&\leq &f(k, s),
\end{eqnarray*}
and the following inequations hold.
\begin{equation*}
f(k-1, s-1)\le \frac{5}{16}f(k, s),
\tag{3}
\end{equation*}
 \begin{equation*}
f(k-2, s-1)\le \frac{1}{8}f(k, s),
\tag{4}
\end{equation*}
\begin{equation*}
f(k, s-1)+f(k-1, s-1)\leq f(k, s).
\tag{5}
\end{equation*}

Now we prove Theorem~\ref{Thm:K4+K3}.

\begin{proof}
We first prove that $gr_k(K_{3}:s\cdot K_{4}^{+}, (k-s)\cdot K_{3})\geq f(k, s)+1$ by constructing a Gallai $k$-colored complete graph with order $f(k, s)$ which contains neither monochromatic copy of $K_{4}^{+}$ colored by one of the first $s$ colors nor monochromatic copy of $K_{3}$ colored by one of the remaining $k-s$ colors.
For this construction, we use the sharpness example of classical Ramsey results.
Let $Q_{1}=C_{(K_{3}, K_{3})}$, $Q_{2}=C_{(K_{4}^{+}, K_{4}^{+})}$ and $Q_{3}=C_{(K_{4}^{+}, K_{3})}$.
Then by Lemma~\ref{Lem:k4+K3}, $Q_{1}$ is a 2-edge colored $K_{5}$, $Q_{2}$ is a 2-edge colored $K_{17}$ and $Q_{3}$ is a 2-edge colored $K_{8}$.
We construct our sharpness example by taking blow-ups of these sharpness examples $Q_{i}~ (i\in\{1, 2, 3\})$. A \emph{blow-up} of an edge-colored graph $G$ on a graph $H$ is a new graph obtained from $G$ by replacing each vertex of $G$ with $H$ and replacing each edge $e$ of $G$ with a monochromatic complete bipartite graph $(V(H), V(H))$ in the same color with $e$.
By induction, suppose that we have constructed $G_{0}, G_{1}, \cdots, G_{i}$, where $G_{0}$ be a single vertex and $G_{i}$ is colored by the color set $[i]$ such that $G_{i}$ has no monochromatic copy of $K_{4}^{+}$ colored by one of the first $s$ colors and no monochromatic copy of $K_{3}$ colored by one of the remaining $i-s$ colors.
If $i=k$, then the construction is completed.
Otherwise, we consider the following cases.

{\bf Case a.} If $i\leq s-2$, then construct $G_{i+2}$ by a blow-up of $Q_{2}$ colored by the two colors $i+1$ and $i+2$ on $G_{i}$. Then $G_{i+2}$ has no monochromatic copy $K_{4}^{+}$ in the first $i+2$ colors.

{\bf Case b.} If $i= s-1$ and $k=s$, then construct $G_{i+1}$ by a blow-up of $K_{4}$ colored by the color $i+1$ on $G_{i}$. Then $G_{i+1}$ has no monochromatic copy $K_{4}^{+}$ in the first $i+1$ colors.

{\bf Case c.} If $i= s-1$ and $k>s$, then construct $G_{i+2}$ by a blow-up of $Q_{3}$ colored by the two colors $i+1$ and $i+2$ on $G_{i}$. Then $G_{i+2}$ contains neither monochromatic copy of $K_{4}^{+}$ colored by one of the first $i+1$ colors nor monochromatic copy of $K_{3}$ colored by the color $i+2$.

{\bf Case d.} If $i\geq s$ and $i=k-1$, then construct $G_{i+1}$ by a blow-up of $K_{2}$ colored by the color $i+1$ on $G_{i}$. Then $G_{i+1}$ contains neither monochromatic copy of $K_{4}^{+}$ colored by one of the first $s$ colors nor monochromatic copy of $K_{3}$ colored by one of the remaining $k-s$ colors.

{\bf Case e.} If $i\geq s$ and $i\leq k-2$, then construct $G_{i+2}$ by a blow-up of $Q_{1}$ colored by the two colors $i+1$ and $i+2$ on $G_{i}$. Then $G_{i+2}$ contains neither monochromatic copy of $K_{4}^{+}$ colored by one of the first $s$ colors nor monochromatic copy of $K_{3}$ colored by one of the remaining $i+2-s$ colors.

By the above construction, it is clear that $G_{k}$ is Gallai $k$-colored and contains neither monochromatic copy of $K_{4}^{+}$ in any of the first $s$ colors nor monochromatic copy of $K_{3}$ in any of the remaining $k-s$ colors.
If $s$ and $k-s$ are both even, then by Case a, we can first construct $G_{s}$ of order $17^{\frac{s}{2}}$. Next by Case e, we continue to construct $G_{s+2}, G_{s+4}, \cdots, G_{k}$. So we can get $|G_{k}|=f_{1}(k, s)$.
If $s$ is even and $k-s$ is odd, then by Case a, we can first construct $G_{s}$ of order $17^{\frac{s}{2}}$. Next by Case e, we continue to construct $G_{s+2}, G_{s+4}, \cdots, G_{k-1}$ and we can get $|G_{k-1}|=17^{\frac{s}{2}}\cdot5^{\frac{k-s-1}{2}}$.
Finally by Case d, we can construct $G_{k}$ of order $f_{2}(k, s)$. Similarly, we can get that $|G_{k}|=f(k, s)$ in the remaining three cases.
Therefore, $$gr_k(K_{3}:s\cdot K_{4}^{+}, (k-s)\cdot K_{3})\geq f(k, s)+1.$$

Now we prove that $gr_k(K_{3}:s\cdot K_{4}^{+}, (k-s)\cdot
K_{3})\leq f(k, s)+1$ by induction on $k+s$. Let $n=f(k, s)+1$ and
$G$ be a Gallai $k$-colored complete graph of order $n$. The
statement is trivial in the case that $k=1$. The statement holds
in the case that $k=2$ by Lemma~\ref{Lem:k4+K3}. The statement
holds in the case that $s=0$ by Theorem~\ref{Thm:K3}. So we can
assume that $s\geq1$ and $k\geq3$, and the statement holds for any
$s^{'}$ and $k^{'}$ such that $s^{'}+k^{'}<s+k$. Then $f(k, s)\geq 16$ and $n\ge 17$.
Suppose, to the contrary, that $G$ contains neither monochromatic
copy of $K_{4}^{+}$ in any of the first $s$ colors nor
monochromatic copy of $K_{3}$ in any of the remaining $k-s$
colors. By Theorem~\ref{Thm:G-Part}, there exists a
Gallai-partition of $V(G)$. Choose a Gallai-partition with the
smallest number of parts, say $(V_{1},
V_{2}, \cdots, V_{t})$ and let $H_{i}=G[V_{i}]$ for $i\in [t]$. Then $t\ge 2$. Choose
one vertex $h_{i}\in V_{i}$ and set $R=G[\{h_{1}, h_{2}, \cdots,
h_{t}\}]$. Then $R$ is a reduced graph of $G$ colored by at most
two colors.

We first consider the case that $t=2$. W.L.O.G, suppose that the
color on the edges between two parts is red. If red is in the last
$k-s$ colors (so $s<k$), then $H_{1}$ and $H_{2}$ both have no red
edges, otherwise, there exists a red $K_{3}$, a contradiction.
Hence $|H_{i}|\leq gr_{k-1}(K_{3}:s\cdot K_{4}^{+}, (k-s-1)\cdot
K_{3})-1$ for each $i\in [2]$. By the induction hypothesis and
(1),
$$|G|=|H_{1}|+|H_{2}|\leq 2f(k-1, s)\leq f(k, s)<n,$$
a contradiction. If red is in the first $s$ colors (W.L.O.G,
suppose that red is the $s$th color), then $H_{1}$ has no red
edges or $H_{2}$ has no red edges, otherwise, there exists a red
$K_{4}^{+}$ in $G$, a contradiction. W.L.O.G, suppose that $H_{1}$
has no red edges. Then $|H_{1}|\leq gr_{k-1}(K_{3}:(s-1)\cdot
K_{4}^{+}, (k-s)\cdot K_{3})-1$. If $H_{2}$ has a red $K_{3}$,
then we can find a red $K_{4}^{+}$ in $G$, a contradiction. Hence
$H_{2}$ contains no monochromatic copy of $K_{4}^{+}$ in any of
the first $s-1$ colors and no monochromatic copy of $K_{3}$ in any
of the remaining $(k-s+1)$ colors. Hence $|H_{2}|\leq
gr_{k}(K_{3}:(s-1)\cdot K_{4}^{+}, (k-s+1)\cdot K_{3})-1$. By the
induction hypothesis and (5), we get that
$$|G|=|H_{1}|+|H_{2}|\leq f(k-1, s-1)+f(k, s-1)\leq f(k, s)<n,$$ a contradiction.

Then we can assume that $t\ge 3$ and since $t$ is smallest, $R$ is colored by exactly two
colors. Suppose that the two colors appeared in the
Gallai-partition $(V_{1}, V_{2}, \cdots, V_{t})$ are red and blue,
that is, the reduced graph $R$ is colored by red and blue. If the
edge $h_{i}h_{j}$ is red in $R$, then $h_{j}$ is said to be a red
neighbor of $h_{i}$. Let $N_{r}(h_{i})=\{$all red neighbors of
$h_{i}\}$, named the red neighborhood of $h_{i}$, and
$N_{b}(h_{i})$ be the blue neighborhood of $h_{i}$ symmetrically.
For each vertex $h_{i} \in V(R)$ $(i\in [t])$, let
$d_{r}(h_{i})=|N_{r}(h_{i})|$, named the red degree of $h_{i}$ in
$R$, and $d_{b}(h_{i})$ be the blue degree of $h_{i}$ in $R$. Then
$d_{r}(h_{i})+d_{b}(h_{i})=t-1$. If there exists one part (say
$V_{1}$) such that all edges joining $V_{1}$ to the other parts
are colored by the same color, then we can find a new
Gallai-partition with two parts $(V_{1}, V_{2}\bigcup\cdots
\bigcup V_{t})$, which contradicts with that $t$ is smallest. It
follows that $t\neq 3$ and the following fact holds.

\begin{fact}\label{fact:dwi}
For any $h_{i} \in V(R)$, we have that $d_{r}(h_{i})\geq1$ and $d_{b}(h_{i})\geq1$.
\end{fact}

Now we can assume that $t\ge 4$. We consider the following three cases.

\begin{case}\label{case:s>=0}
Both red and blue are in the last $k-s$ colors (so $s\leq k-2$).
\end{case}

This means that there is neither red $K_{3}$ nor blue $K_{3}$ in
$G$. Since $R(K_{3}, K_{3})=6$, we have $t\leq5$. By
Fact~\ref{fact:dwi}, $H_{i}$ contains neither red edges nor blue
edges for each $i$. Then $|H_{i}|\leq gr_{k-2}(K_{3}:s\cdot
K_{4}^{+}, (k-s-2)\cdot K_{3})-1$ for $1\leq i\leq t$. By the
induction hypothesis and (2), we get that
\begin{eqnarray*}
|G|=\sum_{i=1}^{t}|H_{i}| \leq t\cdot f(k-2, s) \leq 5f(k-2,
s)\leq f(k, s) < n,
\end{eqnarray*}
a contradiction.

Now we only consider the case that red or blue are in the first
$s$ colors in the following. Then we have the following
Facts.

\begin{fact}\label{fact:red k4} If red is in the first
$s$ colors, then for any $p\in [t-1]$ and any $p$ parts $V_{j_{1}},
\cdots, V_{j_{p}}$, $G[V_{j_{1}}\cup\cdots\cup V_{j_{p}}]$ has no
red $K_{4}$, and the statement holds for blue symmetrically.
\end{fact}

Otherwise, suppose that there is a red $K_{4}$ in
$G[V_{j_{1}}\cup\cdots\cup V_{j_{p}}]$. If there exists one red
edge in $E_{G}(V_{j_{1}}\cup\cdots\cup V_{j_{p}}, V(G)\setminus
(V_{j_{1}}\cup\cdots\cup V_{j_{p}}))$, then $G$ contains a red
$K_{4}^{+}$, a contradiction. It follows that all edges in
$E_{G}(V_{j_{1}}\cup\cdots\cup V_{j_{p}}, V(G)\setminus
(V_{j_{1}}\cup\cdots\cup V_{j_{p}}))$ are blue. Then
$(V_{j_{1}}\cup\cdots\cup V_{j_{p}}, V(G)\setminus
(V_{j_{1}}\cup\cdots\cup V_{j_{p}}))$ is a new Gallai-partition
which contradicts with that $t\geq4$ and $t$ is smallest.

By Fact~\ref{fact:dwi} and Fact~\ref{fact:red k4}, we have the following facts.
\begin{fact}\label{new}
If red is in the first $s$ colors, then every $H_{i}$ has no
red $K_3$, and the statement holds
for blue symmetrically.
\end{fact}

\begin{fact}\label{fact:red k3}
If red is in the first $s$ colors and $h_{i}$ is contained
in a red $K_{3}$ of $R$, then $H_{i}$ has no
red edge, and the statement holds
for blue symmetrically.
\end{fact}

\begin{fact}\label{fact:red edges}
If red is in the first
$s$ colors, and $H_{i}$ and $H_{j}$ both contain red edges, then the edges $E_{G}(V_{i}, V_{j})$ are blue, that is, edge $h_{i}h_{j}$ is blue in $R$, and the statement holds for blue symmetrically.
\end{fact}

Let $I_{r}=\{h_{i}\in V(R)$ : $H_{i}$ contains at least one red
edge$\}$ and $I_{b}=\{h_{i}\in V(R)$ : $H_{i}$ contains at least one
blue edge$\}$. Clearly by Fact~\ref{fact:red edges}, if red is in the first
$s$ colors, then the induced subgraph of $R$ by $I_{r}$ is a blue complete graph and if blue is in the first
$s$ colors, then the induced subgraph of $R$ by $I_{b}$ is a red complete graph.

\begin{fact}\label{fact:VbVr} If red and blue are both in the first $s$
colors, then $|I_{b}\cap I_{r}|\leq 1$.
\end{fact}

Suppose, to the contrary, that $h_{i}$, $h_{j}$$\in I_{b}\cap
I_{r}$. We know that $h_{i}h_{j}$ is either red or blue. Then we
can find either a red $K_{4}$ or a blue $K_{4}$ in $G[V_{i}\cup V_{j}]$,
which contradicts with Fact~\ref{fact:red k4}.

Suppose that $|H_{i}|\geq 2$ for each $i\in [l]$ and $|H_{j}|=1$
for each $j$ with $l+1\leq j\leq t$. Then $l\geq1$. Otherwise,
$R=G$. Then $G$ is 2-edge colored, which contradicts with that
$k\geq3$. Let $p_{0}$ be the number of $H_{i}$($i\in [l]$) which
contains neither red nor blue edges, $p_{1}$ the number of $H_{i}$
which contains either red or blue edges and $p_{2}$ the number of
$H_{i}$ which contains both red and blue edges. Then
$l=p_{0}+p_{1}+p_{2}$, $p_{2}=|I_{b}\cap I_{r}|$ and
$p_{1}+p_{2}=|I_{b}\cup I_{r}|$.

\begin{case}\label{case:s>=1}
Exactly one of red and blue is in the first $s$ colors and the other is in the last $k-s$ colors. (so $s\leq k-1$).
\end{case}

W.L.O.G., suppose that red appears in the first $s$ colors, blue
appears in the last $k-s$ colors and red is the $s$th color. This
means that there is neither red $K_{4}^{+}$ nor blue $K_{3}$ in
$G$. Then by
Fact~\ref{fact:dwi}, we can get that every $H_{i}$ contains no
blue edge, which implies that $p_{2}=0$. Since $R(K_{4}^{+}, K_{3})=9$, we have that $t\leq8$.

\begin{claim}\label{claim:p1}
$p_{1}\leq 2$.
\end{claim}
Otherwise, suppose that $p_{1}\geq 3$. Then W.L.O.G., suppose that $H_{1}$, $H_{2}$ and $H_{3}$ are three corresponding subgraphs which contain red edges.
By Fact~\ref{fact:red edges}, $R[\{h_{1}, h_{2}, h_{3}\}]$ is a blue $K_{3}$. So $G$ has a blue $K_{3}$, a contradiction.

\begin{claim}\label{claim:G}
$$
|G|\leq\left(\frac{1}{8}p_{0}+\frac{5}{16}p_{1}\right)f(k, s)+(t-l).
$$
\end{claim}

\begin{proof}
First let $i\in [l]$ and $H_{i}$ contain no red edges. Then $H_{i}$ is colored with exactly $k-2$ colors and satisfies that it contains neither monochromatic $K_{4}^{+}$ in one of the first $s-1$ color nor monochromatic $K_{3}$ in one of the remaining $k-s-1$ colors.
Hence by the induction hypothesis and (4),
\begin{eqnarray*}
|H_{i}|
&\leq gr_{k-2}(K_{3}:(s-1)\cdot K_{4}^{+}, (k-s-1)\cdot K_{3})-1\\
&= f(k-2, s-1)
\leq\frac{1}{8}~f(k, s)\label{con:1}.
\end{eqnarray*}
Next suppose that $H_{i}$ contains a red edge.
Then $H_{i}$ is colored with exactly $k-1$ colors and by Fact~\ref{new}, $H_{i}$ contains neither monochromatic $K_{4}^{+}$ in one of the first $s-1$ colors nor monochromatic $K_{3}$ in one of the remaining $k-s$ colors.
Hence by the induction hypothesis and (3),
\begin{eqnarray*}
|H_{i}|\leq gr_{k-1}(K_{3}:(s-1)\cdot K_{4}^{+}, (k-s)\cdot K_{3})-1=f(k-1, s-1)\leq\frac{5}{16}~f(k, s)\label{con:2}.
\end{eqnarray*}
Therefore by the above inequalities, we get that
\begin{eqnarray*}
|G|\leq\left(\frac{1}{8}p_{0}+\frac{5}{16}p_{1}\right)f(k, s)+(t-l).
\end{eqnarray*}
\end{proof}

We now consider subcases based on the values of $l$ and $t$.
\begin{subcase}\label{case:t<=5}
$t\leq 5$.
\end{subcase}
By Claim~\ref{claim:p1}, $p_1\le 2$. Since $p_0+p_1=l\le t$ and by
Claim~\ref{claim:G}, we have that
\begin{eqnarray*}
|G|\leq\left(\frac{1\times3}{8}+\frac{5\times2}{16}\right)f(k, s)
 < n,
\end{eqnarray*} a contradiction.

\begin{subcase}\label{case:l<=3}
$l\leq 2$ and $6\leq t\leq 8$.
\end{subcase}
Then by Claim~\ref{claim:p1} and Claim~\ref{claim:G}, we can get
that
\begin{eqnarray*}
|G|
&\leq&\left(\frac{1\times0}{8}+\frac{5\times2}{16}\right)f(k, s)+(t-2)\\
&\leq&\frac{5}{8}f(k, s)+6\\
&\leq&f(k, s) < n,
\end{eqnarray*} a contradiction.

\begin{subcase}\label{case:l>=4}
$l\geq 3$ and $6\leq t\leq 8$.
\end{subcase}

\begin{claim}\label{claim:3p1}
In this case, $p_{1}\leq1$. Further, if $t\ge 7$, then $p_{1}=0$.
\end{claim}

\begin{proof}
First, to the contrary, we can assume
that $H_{1}$ and $H_{2}$ both contain a red edge, and each other
$H_{i}$ contains no red edge by Claim~\ref{claim:p1}. Then by
Fact~\ref{fact:red edges}, we have that the edges
in $E_{G}(V_{1},V_{2})$ are blue. If there are two blue edges in $\{h_{1}h_{i}|\ \ 3\le i\le
t\}$, then we can find a blue triangle $h_{1}h_ph_q$($2\le p<q\le t$) which contradicts with that $G$ has no blue $K_{3}$ or find a red $K_3$ which is $h_2h_ph_q$($3\le p<q\le t$), which contradicts with Fact~\ref{fact:red k3}.
 So we can assume that all edges in
$\{h_{1}h_{i}|\ \ 4\le i\le t\}$ are red. It follows that either
we can find a red triangle $h_1h_ph_q$($4\le p<q\le
t$) which contradicts with Fact~\ref{fact:red k3}, or we can find
a blue $K_3$ which is $h_4h_5h_6$, a contradiction. Secondly, let
$t\ge 7$. To the contrary, we can assume that
$H_{1}$ contains a red edge and each other $H_{i}$ contains no red
edge. Then by Fact~\ref{fact:red k3} and since $G$ has no blue
$K_{3}$, there are at most two red edges in $\{h_{1}h_{i}|\ \ 2\le
i\le t\}$. So we can assume that all edges in $\{h_{1}h_{i}|\ \
4\le i\le t\}$ are blue. It follows that either we can find a blue
$K_3$ which is $h_1h_ph_q$, where $4\le p<q\le t$, a
contradiction, or we can find a red $K_4$ induced by
$\{h_4,h_5,h_6,h_7\}$, which contradicts with Fact~\ref{fact:red
k4}.
\end{proof}

Thus, if $t=6$, then $p_{1}\leq1$. By Claim~\ref{claim:G}, we have that
\begin{eqnarray*}
|G|
&\leq&\left(\frac{5\times1}{16}+\frac{l-1}{8}\right)f(k, s)+(6-l)\\
&=&\frac{2l+3}{16}f(k, s)+(6-l)<f(k, s)<n,
\end{eqnarray*} a contradiction.

If $ t\ge 7$, then by Claim~\ref{claim:3p1} and
Claim~\ref{claim:G}, we have that $p_{1}=0$ and
\begin{eqnarray*}
|G|\leq \frac{l}{8}f(k, s)+(t-l)
\leq f(k, s)
 < n,
\end{eqnarray*} a contradiction. The proof of Case~\ref{case:s>=1} is completed.

\begin{case}\label{case:s>=2}
Red and blue are both in the first $s$ colors (so $s\geq 2$).
\end{case}
This means that there exists neither red $K_{4}^{+}$ nor blue $K_{4}^{+}$ in
$G$. W.L.O.G., suppose that red and blue are the $(s-1)$th and
$s$th color, respectively. Since $R(K_{4}^{+}, K_{4}^{+})=18$, we know that
$t\leq 17$. Since $s\ge 2$ and $k\ge 3$, $f(k, s)\ge 35$. First we prove some claims.

\begin{claim}\label{claim:drbwi}
For any $h_{i} \in V(R)$, we have $d_{r}(h_{i})\leq8$ and $d_{b}(h_{i})\leq8$ in $R$.
\end{claim}

W.L.O.G., suppose, to the contrary, that $d_{r}(h_{i})\geq9$. Since $R(K_{3},
K_{4}^{+})=9$, then the subgraph of $R$ induced by $N_{r}(h_{i})$
contains either a red $K_{3}$ or a blue $K_{4}^{+}$. So $R$
contains a red  $K_{4}^{+}$ or blue $K_{4}^{+}$. Thus $G$ contains
also, a contradiction.

\begin{claim}\label{claim:drbwi2}
If $d_{r}(h_{i})\geq4$, then $h_{i}\notin I_{r}$, and if $d_{b}(h_{i})\geq4$, then $h_{i}\notin I_{b}$.
\end{claim}

Suppose that $d_{r}(h_{i})\geq4$. To the contrary, suppose that
$H_{i}$ contains a red edge. If the induced subgraph
$R[N_{r}(h_{i})]$ contains a red edge, say $h_{p}h_{q}$, then we
find a red $K_3$ which is $h_{i}h_{p}h_{q}$, which contradicts
with Fact~\ref{fact:red k3}. Otherwise, $R[N_{r}(h_{i})]$ contains
a blue $K_{4}$. So $G[\bigcup_{h_p\in N_{r}(h_{i})}V_{p}]$ contains
a blue $K_{4}$, which contradicts with Fact~\ref{fact:red k4}. So we have
that $h_{i}\notin I_{r}$. The proof for blue is as same as the above one for red
symmetrically.

\begin{claim}\label{claim:Vb+Vr}
$|I_{b}|+|I_{r}|\leq 4$.
\end{claim}

If there is a vertex $h_{i}\in I_{b}\cap I_{r}$, then by
Fact~\ref{fact:red k3}, $h_{i}$ is contained in neither a red
$K_{3}$ nor a blue $K_{3}$ in $R$. By Fact~\ref{fact:red edges},
we know that the induced subgraph of $R$ by $I_{r}$ is a blue
complete graph and the induced subgraph of $R$ by $I_{b}$ is a red
complete graph. It follows that $|I_{b}|\leq 2$ and $|I_{r}|\leq
2$.

Now we can assume that $I_{b}\cap I_{r}=\emptyset$ by
Fact~\ref{fact:VbVr}. To the contrary, suppose that
$|I_{b}|+|I_{r}|\ge 5$. If $|I_{b}|\geq 4$, then by
Fact~\ref{fact:red edges}, the subgraph $R[I_{b}]$ contains a red
$K_{4}$, which contradicts with Fact~\ref{fact:red k4}. Then
$|I_{b}|\leq 3$. By the same reasons, we know that $|I_{r}|\leq
3$.  Then $|I_{b}|=3$ or $|I_{r}|=3$. W.L.O.G., suppose that
$|I_{b}|=3$. Then $|I_{r}|\ge 2$. Let $I_{b}=\{h_1, h_2, h_3\}$
and $h_4, h_5\in I_{r}$. By Fact~\ref{fact:red edges}, $h_1h_2h_3$
is a red triangle and $h_4h_5$ is a blue edge in $R$. It is easy
to check that there exists a red triangle $h_ph_ih_j$ such that $p\in
\{4, 5\}$ and $i,j\in [3]$ or a blue $K_3=h_ih_4h_5$ such that
$i\in [3]$, which contradicts with Fact~\ref{fact:red k3}.

\begin{claim}\label{claim:G2}
$$
|G|\leq \left(\frac{1}{17}p_{0}+\frac{5}{2\times17}p_{1}+\frac{5}{17}p_{2}\right)f(k, s)+(t-l).
$$
\end{claim}

\begin{proof}
First suppose that $H_{i}$($i\leq l$) contains neither red nor blue edges.
This means that $H_{i}$ is colored with exactly $k-2$ colors and satisfies that it has neither monochromatic $K_{4}^{+}$ in one of the first $s-2$ colors nor monochromatic $K_{3}$ in one of the remaining $k-s$ colors.
It is easy to check that $$\frac{f_{j}(k-2, s-2)}{f_{j}(k, s)}=\frac{1}{17},$$ for any $1\leq j\leq 5$.
Hence by the induction hypothesis,
\begin{eqnarray*}|H_{i}|\leq gr_{k-2}(K_{3}:(s-2)\cdot K_{4}^{+}, (k-s)\cdot K_{3})-1=f(k-2, s-2)=\frac{1}{17}f(k, s).\ \ \ \ \ \ \ \ \ \ \ \ \ \ \ \ \ \ \ \ \ \ \ \ \ (6)\end{eqnarray*}
Next suppose that $H_{i}$ contains no red edges but contains blue edges.
This means that $H_{i}$ is colored with exactly $k-1$ colors and satisfies that it contains neither monochromatic $K_{4}^{+}$ in one of the first $s-2$ colors nor monochromatic $K_{3}$ in one of the remaining $k-s+1$ colors by Fact~\ref{new}.
It is easy to check that $$\frac{f_{2}(k-1, s-2)}{f_{1}(k, s)}=\frac{f_{3}(k-1, s-2)}{f_{3}(k, s)}=\frac{f_{5}(k-1, s-2)}{f_{4}(k, s)}=\frac{2}{17},$$ $$\frac{f_{1}(k-1, s-2)}{f_{2}(k, s)}=\frac{f_{4}(k-1, s-2)}{f_{5}(k, s)}=\frac{5}{2\times17},$$
So by the induction hypothesis,
\begin{eqnarray*}|H_{i}|\leq gr_{k-1}(K_{3}:(s-2)\cdot K_{4}^{+}, (k-s+1)\cdot K_{3})-1=f(k-1, s-2)\leq\frac{5}{2\times17}f(k, s)\label{con:4}. \ \ \ \ \ \ \ \ \ \ \ (7)\end{eqnarray*}
The same inequality holds if $H_{i}$ contains no blue edges but contains red edges.

Finally suppose that $H_{i}$ contains both red and blue edges.
This means that $H_{i}$ is colored with all $k$ colors and satisfies that it contains neither monochromatic $K_{4}^{+}$ in one of the first $s-2$ colors nor monochromatic $K_{3}$ in one of the remaining $k-s+2$ colors by Fact~\ref{new}.
It is easy to check that $\frac{f(k, s-2)}{f(k, s)}\leq\frac{5}{17},$
so by the induction hypothesis,
\begin{eqnarray*}|H_{i}|\leq gr_{k}(K_{3}:(s-2)\cdot K_{4}^{+}, (k-s+2)\cdot K_{3})-1=f(k, s-2)\leq\frac{5}{17}f(k, s)\label{con:5}. \ \ \ \ \ \ \ \ \ \ \ \ \ \ \ \ \ \ \ \ \ \ \ \ \ \ (8)\end{eqnarray*}

Combining Inequalities(6)-(8), we have the following inequality
\begin{eqnarray*}
|G|\leq \left(\frac{1}{17}p_{0}+\frac{5}{2\times17}p_{1}+\frac{5}{17}p_{2}\right)f(k, s)+(t-l).
\end{eqnarray*}
\end{proof}

We now consider subcases based on the value of $l$ and $t$.

\begin{subcase}\label{case:13=<t<=17}
$13\leq t\leq 17$.
\end{subcase}

By Claim~\ref{claim:drbwi}, $d_{r}(h_{i})\leq8$ and $d_{b}(h_{i})\leq8$ in $R$ for any $i\in [l]$.
Since $d_{r}(h_{i})+d_{b}(h_{i})=t-1$, we have that $d_{b}(h_{i})\geq4$ and $d_{r}(h_{i})\geq4$ in $R$.
Then by Claim~\ref{claim:drbwi2}, every $H_{i}$ contains neither red nor blue edge.
So $p_{2}=p_{1}=0$. Thus $p_{0}=l$. By Claim~\ref{claim:G2}, we have that
\begin{eqnarray*}
|G|\leq \frac{l}{17}~f(k, s)+(t-l)< f(k, s)+1=n,
\end{eqnarray*}
a contradiction.

Next, we consider the case that $4\leq t\leq 12$. By
Fact~\ref{fact:VbVr}, we have that $p_{2}\leq1$.

\begin{subcase}\label{case:1=<l<=3}
$l\leq 3$.
\end{subcase}
Then $p_{1}\leq2$ if $p_{2}=1$ and $p_{1}\leq3$ if $p_{2}=0$. Hence by Claim~\ref{claim:G2}, we get
\begin{eqnarray*}
|G|
&\leq&\begin{cases}
\frac{10}{17}~f(k, s)+(t-3), &\text{if $p_{2}=1$ and $p_{1}\leq2$,}\\
\frac{15}{34}~f(k, s)+(t-3), &\text{if $p_{2}=0$ and $p_{1}\leq3$}\\
\end{cases}\\
&\leq &f(k, s)< n,
\end{eqnarray*}
a contradiction.

\begin{subcase}\label{case:4=<l<=10}
$4\leq l\leq10$.
\end{subcase}

First suppose that $p_{2}=1$. It follows that $t\leq7$. Otherwise, for any $h_{i} \in V(R)$, we have that $d_{r}(h_{i})\ge 4$ or $d_{b}(h_{i})\ge 4$ in $R$ since $d_{r}(h_{i})+d_{b}(h_{i})=t-1$. Then by
Claim~\ref{claim:drbwi2}, every $H_{i}$ contains no red
edge or contains no blue edge, which contradicts with the assumption that $p_{2}=1$. By Claim~\ref{claim:Vb+Vr}, we have that $p_{1}\leq2$. Thus, by
Claim~\ref{claim:G2}, we have that
\begin{eqnarray*}
|G|\leq \frac{14}{17}~f(k, s)< n,
\end{eqnarray*} a contradiction.

Now we assume that $p_{2}=0$. By Claim~\ref{claim:Vb+Vr}, we have that $p_{1}\leq4$.
This means that
\begin{eqnarray*}
|G|\leq \left[\frac{5}{2\times17}p_{1}+\frac{1}{17}(l-p_{1})\right]f(k, s)+(t-l)
\leq\frac{16}{17}~f(k, s)+2\leq f(k, s)
 < n,
\end{eqnarray*} a contradiction.

\begin{subcase}\label{case:l=11}
$l\ge 11$.
\end{subcase}

Then $11\le l\le t\le 12$. Hence $d_{r}(h_{i})\geq4$ or
$d_{b}(h_{i})\geq4$ for each $i\in [t]$. It follows that $p_{2}=0$
by Claim~\ref{claim:drbwi2}. If $d_{r}(h_{i})\geq4$ and
$d_{b}(h_{i})\geq4$ for each $i\in [l]$, then $p_{1}=0$ by
Claim~\ref{claim:drbwi2}. So $p_{0}=l$. By Claim~\ref{claim:G2},
we have that $$|G|\leq\frac{l}{17}f(k, s)+(t-l)\leq f(k, s)<n,$$ a
contradiction. So, W.L.O.G., we can assume that
$d_{r}(h_{1})\geq4$ and $d_{b}(h_{1})\le 3$. By
Claim~\ref{claim:drbwi}, we know that $(d_{r}(h_{1}),
d_{b}(h_{1}), l, t)\in \{(7, 3, 11, 11), (8, 2, 11, 11), (8, 3,
11, 12), (8, 3, 12, 12)\}$. Let
$\widetilde{d_{r}}(h_{1})=|N_r(h_{1})\cap \{h_{1}, \cdots,
h_{l}\}|$ and $\widetilde{d_{b}}(h_{1})=|N_b(h_{1})\cap \{h_{1},
\cdots, h_{l}\}|$. Then $(\widetilde{d_{r}}(h_{1}),
\widetilde{d_{b}}(h_{1}), l)\in \{(7, 3, 11), (8, 2, 11), (8, 3,
12)\}$. Let $F$ be the subgraph of $R$ induced by $N_r(h_{1})\cap
\{h_{1}, \cdots, h_{l}\}$. So $|F|\ge 7$. Clearly, $F$ has no red
$K_3$. Otherwise we can find a red $K_4$ obtained by a red $K_3$
in $F$ and $h_1$, which contradicts with Fact~\ref{fact:red k4}.
\begin{claim}\label{claim:FIrIb}
 $V(F)\cap I_{b}=\emptyset$, $V(F)\cap I_{r}=\emptyset$
and $h_1\notin I_r$.
\end{claim}
First we claim that the red degree of $h_i$ in $F$ is at most 3
for each $h_i\in V(F)$. Otherwise, suppose that $h_i$ has at least
four red neighbors in $F$. Since $F$ has no red $K_3$, we can find
a blue $K_4$ induced by the red neighbors of $h_{i}$ in $F$, which
contradicts with Fact~\ref{fact:red k4}. Then the blue degree
of $h_i$ in $F$ is at least 3 for each $h_i\in V(F)$. It follows
that every vertex $h_i$ of $F$ is contained in a blue triangle of
$F$.
 Thus, by
Fact~\ref{fact:red k3}, for each $h_i\in V(F)$, we have that
$h_{i}\notin I_{b}$. Secondly, we claim that the blue degree of
$h_i$ in $F$ is at most 5 for each $h_i\in V(F)$. Otherwise, we
find a blue $K_3$ induced by the blue
neighbors of $h_{i}$ in $F$ since $R(K_{3}, K_{3})=6$ and $F$ has no red $K_3$. Then there
exists a blue $K_{4}$ in $R$, which contradicts with
Fact~\ref{fact:red k4}. It follows that the red degree of $h_i$ in
$F$ is at least 1 for each $h_i\in V(F)$. Hence $h_1$ and $h_i$
are contained in the same red triangle of $R$ for each vertex
$h_i$ of $F$. Thus, by Fact~\ref{fact:red k3}, $h_{1}\notin I_{r}$
and for each $h_{i}\in V(F)$, we have that $h_{i}\notin I_{r}$.

 By Claim~\ref{claim:FIrIb}, $p_{0}\geq |F|$. Now we consider the case that $(\widetilde{d_{r}}(h_{1}),
\widetilde{d_{b}}(h_{1}, l)=(8, 2, 11)$. Then $p_{0}\geq 8$. By Fact~\ref{fact:VbVr},
$p_2\le 1$.

If $p_2=0$, then by Claim~\ref{claim:G2}, we get that
\begin{eqnarray*}
|G|\leq
\left[\frac{1}{17}p_{0}+\frac{5}{2\times17}(11-p_{0})\right]f(k,
s)+(t-l) \leq\frac{31}{34}~f(k, s)+1< f(k, s)
 < n,
\end{eqnarray*} a contradiction. If $p_2=1$, we can assume that
$h_p\in I_r\cap I_b$. Then by Claim~\ref{claim:FIrIb}, $h_p\in
N_b(h_{1})$. By Fact~\ref{fact:red k4}, $h_1\notin I_b$. Then by
Claim~\ref{claim:FIrIb}, $h_1\notin I_r\cup I_b$. So
$p_{0}\geq |F|+1=9$. Then by Claim~\ref{claim:G2}, we get that
\begin{eqnarray*}
|G|\leq
\left[\frac{9}{17}+\frac{5}{2\times17}+\frac{5}{17}\right]f(k, s)+(t-l)
\leq\frac{33}{34}~f(k, s)+1< f(k, s)+1= n,
\end{eqnarray*} a contradiction.

Now we consider the remaining cases that
$(\widetilde{d_{r}}(h_{1}), \widetilde{d_{b}}(h_{1}), l)\in \{(7,
3, 11), (8, 3, 12)\}$. Then we can assume that $N_b(h_{1})\cap
\{h_{1}, \cdots, h_{l}\}=\{h_o, h_p, h_q\}$. First consider the
case that there is a blue edge spanned by vertices in $\{h_o, h_p,
h_q\}$, say $h_ph_q$. Then $h_{1}h_ph_q$ is a blue $K_{3}$. By
Fact~\ref{fact:red k3}, we have that $h_{1}, h_p, h_q\notin
I_{b}$. Then by Claim~\ref{claim:FIrIb}, $h_1\notin I_r\cup I_b$. So $p_{0}\geq |F|+1$. If $p_2=0$, then by
Claim~\ref{claim:G2}, we get that
\begin{eqnarray*}
|G|\leq
\left[\frac{1}{17}p_{0}+\frac{5}{2\times17}(l-p_{0})\right]f(k,
s)+(t-l) \leq\frac{33}{34}~f(k, s)+1< f(k, s)+1= n,
\end{eqnarray*} a contradiction. If $p_2=1$, then
$h_o\in I_r\cap I_b$ and $h_oh_p, h_oh_q$ are red. Then by
Fact~\ref{fact:red k4}, $h_p, h_q\notin I_r\cup I_b$. So $p_{0}=|F|+3$. Then by Claim~\ref{claim:G2}, we get that
\begin{eqnarray*}
|G|\leq \left[\frac{11}{17}+\frac{5}{17}\right]f(k, s)+t-l
< f(k, s)+1=n,
\end{eqnarray*} a contradiction.

Secondly, we consider the case that $h_oh_ph_q$ is a red $K_{3}$.
By Fact~\ref{fact:red k3}, $h_o, h_p, h_q\notin I_r$. Then
$p_2=0$. By Fact~\ref{fact:red k4}, for each vertex $h_i$ in $F$,
there is at least one blue edge in $E_R(h_i, \{h_o, h_p, h_q\})$.
This means that there are at least 7 blue edges between $V(F)$ and
$\{h_o, h_p, h_q\}$. By the pigeonhole principle, there is a
vertex in $\{h_o, h_p, h_q\}$, say $h_o$, such that
$|N_{b}(h_o)\cap V(F)|\ge 3$.
If the subgraph induced by $N_{b}(h_o)\cap V(F)$ contains no blue
edge, then $N_{b}(h_o)\cap V(F)$ along with $h_{1}$ induce a red
$K_{4}$ in $R$, which contradicts with Fact~\ref{fact:red k4}.
Then the subgraph induced by $N_{b}(h_o)\cap V(F)$ contains a blue
edge. So $h_o$ is contained in both a red $K_{3}$ and a blue
$K_{3}$. Thus $h_o\notin I_{b}\cup I_{r}$ by Fact~\ref{fact:red k3}. So $p_{0}\geq |F|+1$. By
Claim~\ref{claim:G2}, we have that
$$|G|\leq\left[\frac{1}{17}p_{0}+\frac{5}{2\times17}(l-p_{0})\right]f(k, s)+(t-l)\leq\frac{33}{34}~f(k, s)+1< f(k, s)+1= n,$$ a contradiction.

Complete the proof of Case~\ref{case:s>=2} and then the proof of Theorem~\ref{Thm:K4+K3}.

\end{proof}

\bibliography{bibfile1}

\end{document}